\documentclass[10 pt]{amsart}

\usepackage{amsmath} 
\usepackage{amsfonts}
\usepackage{amssymb}
\usepackage{amstext}
\usepackage{amsbsy}
\usepackage{amsopn}
\usepackage{amsthm}
\usepackage{amsxtra}
\usepackage{graphicx}
\usepackage{caption}
\usepackage{subcaption}
\usepackage{color}
\usepackage{hyperref}
\usepackage{enumerate}
\usepackage{amscd}

\newtheorem{theorem}{Theorem}
\newtheorem{lemma}[theorem]{Lemma}
\newtheorem{corollary}[theorem]{Corollary}

\newtheorem*{conjecture*}{Conjecture}
\newtheorem*{claim*}{Claim}
\newtheorem*{theorem*}{Theorem}

\theoremstyle{remark}

\theoremstyle{definition}

\newcommand{\A}{\mathcal{A}}

\newcommand{\Z}{\mathbb{Z}}

\newcommand{\n}{\noindent}
\newcommand{\N}{\mathbb{N}}
\newcommand{\rst}[1]{\ensuremath{{\mathbin\upharpoonright}%
\raise-.5ex\hbox{$#1$}}}  

\newcommand{\Aut}{{\rm Aut}}
\newcommand{\id}{{\rm Id}} 

\title{The automorphism group of a shift of slow growth is amenable}
\author{Van Cyr}
\address{Bucknell University, Lewisburg, PA 17837 USA}
\email{van.cyr@bucknell.edu}
\author{Bryna Kra}
\address{Northwestern University, Evanston, IL 60208 USA}
\email{kra@math.northwestern.edu}


\thanks{The second author was partially supported by NSF grant 1500670.}

\begin{document}

\begin{abstract}
Suppose $(X,\sigma)$ is a subshift, $P_X(n)$ is the word complexity function of $X$, and $\Aut(X)$ is the group of automorphisms of $X$.  We show that if $P_X(n)=o(n^2/\log^2 n)$, then $\Aut(X)$ is amenable (as a countable, discrete group).  We further show that if $P_X(n)=o(n^2)$, then $\Aut(X)$ can never contain a nonabelian free semigroup (and, in particular, can never contain a nonabelian free subgroup).  This is
in contrast to recent examples, due to Salo and Schraudner, of subshifts with quadratic complexity that do contain such a semigroup. 
\end{abstract} 

\maketitle 

\section{Amenability and the automorphism group}

For a subshift $(X, \sigma)$ over a finite alphabet, let $\Aut(X) = \Aut(X, \sigma)$ denote the group 
of all automorphisms of the system, meaning the collection of all homeomorphisms 
$\phi\colon X\to X$ such that $\phi\circ\sigma = \sigma\circ\phi$.  
The automorphism group of many subshifts with positive entropy, including the full shift and more generally any mixing shift of finite type, is a countable group that contains many structures, including isomorphic copies of any finite group, countably many copies of $\Z$, and the free group on countably many generators (see~\cite{H, BLR}). In particular, when given the discrete topology, these automorphism groups are never amenable.  
This behavior is in contrast to what happens in minimal shifts of zero entropy: if the complexity 
function $P_X(n)$, which counts the number of words in the language of the shift, satisfies
$\limsup_{n\to\infty}\frac{\log(P_X(n))}{n^\beta} = 0$ for some $\beta < 1/2$, then the automorphism group $\Aut(X)$ is amenable; 
furthermore, every finitely 
generated torsion-free subgroup of the automorphism group has subexponential growth~\cite{CK}.   
For lower complexities, one can sometimes carry out a more detailed analysis of the automorphism group, 
and this is done for polynomial growth in~\cite{CK}, and with extra assumptions on the dynamics, sometimes one can give a complete description of the automorphism group (see~\cite{CK3, CK2, DDMP}). 

We continue the systematic study of automorphism groups here, focusing on subshifts 
with zero entropy.  These automorphism groups are constrained by the subexponential growth rate of words 
in the language of the shift, and it seems plausible that for any subshift $(X, \sigma)$ of zero entropy, we have a version of 
the Tits alternative:  either $\Aut(X)$ contains a free subgroup or $\Aut(X)$ is amenable.  
It may be possible that a stronger alternative holds, namely either $\Aut(X)$ contains a free subgroup or it is 
virtually abelian.  Somewhat surprisingly, 
we can not rule out that such an alternative holds for any shift, even without an assumption on the entropy. 
For example, this dichotomy holds for any mixing subshift of finite type~\cite{BLR}, since the automorphism group contains the free group on two generators, and it holds for other classes of subshifts with positive entropy, such as Toeplitz systems, where the automorphism group is abelian~\cite{DDMP2}.

Furthermore, a stronger result is plausible, namely that for any zero entropy subshift, the automorphism group is amenable. 
Numerous results support this statement:  the automorphism group is amenable 
for any minimal subshift whose complexity is 
stretched exponential with exponent less than $1/2$, for all subshifts with linear complexity, 
and for several other classes of low complexity shifts (see~\cite{CK3, CK2, CK, DDMP}). 

To address these questions and conjectures, 
we give a detailed analysis of the algebraic properties of $\Aut(X)$ for shifts whose complexity is at most quadratic.  
In~\cite{CK3}, we showed that for a transitive shift with subquadratic growth, after quotienting the automorphism group by the subgroup generated by the shift, we are left with a periodic group.  
This left open a stronger description of this automorphism group, as well as what happens without an assumption of transitivity.  
As a first step in addressing this, we show (see Section~\ref{sec:background} for precise definitions): 
\begin{theorem}
\label{th:main}
Assume that $(X, \sigma)$ is a subshift whose complexity function satisfies $P_X(n) = o(n^2/\log^2n)$.  Then 
$\Aut(X)$ is amenable  (as a countable discrete group). 
\end{theorem}

In particular, the automorphism group of any shift whose complexity is $o(n^{2-\varepsilon})$, for some $\varepsilon> 0$, 
is amenable.

The techniques to prove Theorem~\ref{th:main} follow a basic strategy developed in~\cite{CK}, 
but deducing the theorem without the assumption that $X$ is minimal adds significant technical difficulties.  
One of the new ideas used is the construction of a descending chain of subshifts for which each term retains some of the properties that make minimal shifts easier to study.  We believe that this technique should prove to be applicable in other settings.

Unfortunately, our methods do not easily extend to a shift whose complexity is $o(n^2)$, but in this setting we are able to prove a weaker result that holds for this larger class of shifts: 
\begin{theorem}\label{th:main2} 
Assume that $(X, \sigma)$ is a subshift whose complexity function satisfies $\liminf_{\n\to\infty}\frac{P_X(n)}{n^2} = 0$.  Then $\Aut(X)$ does not contain an isomorphic copy of the free semigroup on two generators. 
\end{theorem}
In particular, such an automorphism group can not contain any nonabelian free subgroup.  
The interest in this theorem is the constrast with a recent result of Salo and Schraudner: they constructed a subshift $(X, \sigma)$ whose complexity function satisfies $P_X(n) = (n+1)^2$ and whose automorphism group contains a free semigroup on two generators.  This highlights the subtle issues that arise in addressing questions and conjectures on amenability 
of $\Aut(X)$ when $(X,\sigma)$ is not minimal, and the difficulty in passing beyond quadratic complexity.

\section{Background} 
\label{sec:background}
\subsection{Subshifts}
Let $\mathcal{A}$ be a finite alphabet and endow $\A^{\Z}$ with  the topology induced by 
the metric 
$$ 
d(x,y):=2^{-\inf\{|i|\colon x_i\neq y_i\}}. 
$$ 
For $x\in\A^{\Z}$, we denote the $i^{th}$ coordinate of $x$ by $x_i$.  For $n\in\N$, an element $w=(w_0,\dots,w_{n-1})\in\A^n$ is called a {\em word of length $n$}.  If $w$ is a word of length $n$, then the set 
$$ 
[w]_0^+:=\{x\in\A^{\Z}\colon x_i=w_i \text{ for all }0\leq i<n\} 
$$ 
is the {\em cylinder set} determined by $w$.  The collection of all cylinder sets is a basis for the topology of $\A^{\Z}$.  The {\em (left) shift} $\sigma\colon\A^{\Z}\to\A^{\Z}$ is the map $x\mapsto\sigma x$ given by $(\sigma x)_i:=x_{i+1}$ 
for all $i\in\Z$, and it is a homeomorphism of $\A^{\Z}$.  If $x\in\A^{\Z}$ and there exists $p>0$ such that $\sigma^px=x$, then $x$ is {\em periodic} of period $p$.  If no such $p$ exists, $x$ is {\em aperiodic}. 

A closed, $\sigma$-invariant subset $X\subset\A^{\Z}$ together with the shift $\sigma\colon X\to X$ 
is called a {\em subshift}.  If $X$ is a subshift, we define the {\em language $\mathcal{L}(X)$} of $X$ to be 
$$ 
\mathcal{L}(X):=\left\{w\in\bigcup_{n=1}^{\infty}\A^n\colon[w]_0^+\cap X\neq\varnothing\right\}. 
$$ 
For $n\in\N$, the set $\mathcal{L}_n(X):=\mathcal{L}(X)\cap\A^n$ denotes the set of {\em words of length $n$} in the language of $X$, and we denote the length of word $w\in\mathcal{L}(X)$ by $|w|$.  

\subsection{Complexity}
The {\em complexity function of $X$} is the function $P_X\colon\N\to\N$ defined by $P_X(n):=|\mathcal{L}_n(X)|$.  If $x\in\A^{\Z}$, then the {\em orbit closure $\mathcal O(x)$ of $x$} under the shift  
$$ 
\overline{\mathcal{O}(x)}:=\overline{\{\sigma^ix\colon i\in\Z\}} 
$$ 
is also a subshift.  We make a slight abuse of notation 
and refer to $P_{\overline{\mathcal{O}(x)}}(n)$ as the {\em complexity function of $x$}.  To avoid confusion, we use the lowercase letter $x$ to refer to an element of $\A^{\Z}$ and the uppercase letter $X$ to refer to a subshift of $\A^{\Z}$.  
The basic result relating  dynamical properties of $x$ to its complexity is the Morse-Hedlund Theorem~\cite{MH}: 
an element $x\in\A^{\Z}$ is aperiodic if and only if its complexity function is bounded below by $n+1$ for all $n$.

Suppose $w=(w_0,\dots,w_{n-1})\in\mathcal{L}_n(X)$ and $L\in\N$ is fixed.  We say that $w$ {\em extends uniquely } $L$ times to the right and left (in the language of $X$) if there is a unique $u=(u_0,\dots,u_{n+2L-1})\in\mathcal{L}_{n+2L}(X)$ such that $w_i=u_{i+L}$ for $0\leq i<n$.  
 If $w\in\mathcal{L}(X)$ and $u\in\mathcal{L}_n(X)$ for some $n\geq|w|$, we say that $w$ is a {\em subword} of $u$ if there exists $0\leq i<n-|w|$ such that $u_j=w_j$ for $i\leq j<|w|$. 
Thus if $w$ extends uniquely $L$ times in $X$, then if $x\in X$ and if $w=(x_j,\dots,x_{j+n-1})$ for some $j\in\Z$, then $u=(x_{j-L},\dots,x_{j+n+L-1})$.  Rephrasing this, whenever $w$ is a word in $x$, then $u$ is also a word in $x$ and $w$ is a subword of $u$. 

\subsection{The automorphism group}
If $(X, \sigma)$ is a subshift and $Hom(X)$ is the group of all homeomorphisms from $X$ to itself, then the {\em group of automorphisms} of $X$, denoted $\Aut(X)$, is the centralizer of $\sigma$ in $Hom(X)$.  (Strictly speaking, we should write $\Aut(X, \sigma)$ instead of $\Aut(X)$, but we assume that the subshift is endowed with the shift and omit explicit mention of $\sigma$ from most of our notation.)  
A function $\varphi\colon X\to X$ is called a {\em sliding block code} if there exists $R\in\N\cup\{0\}$ such that $(\varphi x)_0$ is a function of $(x_{-R},\dots,x_R)$ for all $x\in X$.  In this case, $R$ is called a {\em range} of $\varphi$.  The classical Curtis-Hedlund-Lyndon Theorem~\cite{H} states that every automorphism of $X$ is a sliding block code.  In particular, this means that 
for any subshift $(X, \sigma)$, the automorphism group $\Aut(X)$ is countable. 

For $R\in\N\cup\{0\}$, 
define $\Aut_R(X)$ to be the set of all $\phi\in \Aut(X)$ such that both $\phi$ and its inverse 
are given by sliding block codes of range $R$.  
Since any block code of range $R$ is also a block code of range $R+1$, we have 
$$ 
\Aut_0(X)\subset\Aut_1(X)\subset\Aut_2(X)\subset\Aut_3(X)\subset\cdots 
$$ 
and  $\Aut(X)=\bigcup_{R=0}^{\infty}\Aut_R(X)$.  If $\varphi\in\Aut_R(X)$ and if $w\in\mathcal{L}(X)$ is such that $|w|\geq2R+1$, then we define $\varphi(w)$ to be the word of length $|w|-2R$ obtained by applying the (range $R$) block code defining $\varphi$ to $w$.  Note that this definition is not intrinsic to $\varphi$ but rather to $\varphi$ together with a range $R$.  Whenever we apply an automorphism $\varphi\in\Aut_R(X)$ to a word $w$, we are implicitly choosing $R$ to be the range of $\varphi$. 

For each $w\in\mathcal{L}(X)$ and $n\in\N$, define the function $\mathcal{W}_n\colon([w]_0^+\cap X)\to\mathcal{L}_{|w|+2n}(X)$ by 
$$ 
\mathcal{W}_n(x):=(x_{-n},x_{-n+1},\dots,x_0,x_1,\dots,x_{|w|+n-2},x_{|w|+n-1}). 
$$ 
With this notation, $w$ extends uniquely $n$ times to the right and left (in the language of $X$) if $\mathcal{W}_n(x)$ is independent of $x\in[w]_0^+\cap X$. 

Suppose $u_1,\dots,u_k,v\in\mathcal{L}(X)$.  An automorphism $\phi\in\Aut(X)$ {\em preserves occurrences of $v$} if $\phi([v]_0^+\cap X)\subset[v]_0^+\cap X$.  If $D\in\N$, then $\phi$ {\em preserves occurrences of $v$ when it is $D$ units from $u_1,\dots,u_k$} if for any $x\in[v]_0^+\cap X$ such that $\mathcal{W}_D(x)$ does not contain $u_i$ as a subword for any $1\leq i\leq k$, we have $\phi(x)\in[v]_0^+\cap X$. 

To illustrate the usefulness of this notion, we note the following lemma: 
\begin{lemma}\label{lem:useful} 
Let $R\in\N\cup\{0\}$ be fixed and suppose $\varphi,\psi\in\Aut_R(X)$.  Suppose $w\in\mathcal{L}(X)$ extends uniquely $2R$ times to the right and left and let $\tilde{w}\in\mathcal{L}(X)$ be the unique word obtained by this extension.  If $\varphi(\tilde{w})=\psi(\tilde{w})$, then $\varphi^{-1}\circ\psi$ preserves occurrences of $\tilde{w}$. 
\end{lemma} 
\begin{proof} 
Since $w$ extends uniquely $2R$ times to both sides (to $\tilde{w}$), it suffices to show that $\varphi^{-1}\circ\psi$ preserves occurrences of $w$.  Let $x\in[w]_0^+\cap X$.  By assumption, 
$\mathcal{W}_{2R}(x)=\tilde{w}$.  Since $\varphi$ and $\psi$ are block codes of range $R$, $\mathcal{W}_R(\varphi x)=\varphi(\tilde{w})=\psi(\tilde{w})=\mathcal{W}_R(\psi x)$.  Since $\varphi^{-1}$ is a block code of range $R$ and $\varphi^{-1}(\varphi  x)=x$, 
we have $\varphi^{-1}(\mathcal{W}_R(\varphi x))=w$.  Since $\psi x\in[\mathcal{W}_R(\phi x)]_0^+\cap X$, 
it follows that $(\varphi^{-1}\circ\psi)(x)\in[w]_0^+$.  
\end{proof}

Let $(X,\sigma)$ be a subshift and suppose $w\in\mathcal{L}(X)$.  Let $X(w)\subset X$ denote the (possibly empty) subshift of $X$ obtained by forbidding the word $w$: 
$$ 
X(w):=X\setminus([w]_0^+\cap X)=\left\{x\in X\colon \sigma^jx\notin[w]_0^+ \text{ for all }  j\in\Z\right\}. 
$$ 

\subsection{Amenability}
If $G$ is a group and $F\subset G$, let $|F|$ denote the cardinality of the set $F$, and for $g\in G$ the set $gF$ is defined to be the 
set $\{gf\colon f\in F\}$.  
A discrete, countable group $G$ is {\em amenable}
if there exists a sequence $(F_k)_{k\in\N}$ of finite subsets of $G$ such that 
every $g\in G$ is contained in all but finitely many $F_k$ and such that
$$
\lim_{k\to\infty}\frac{|F_k\Delta gF_k|}{|F_k|} = 0
$$
for all $g\in G$.  In this case, the sequence $(F_k)_{k\in\N}$  is called a {\em F\o lner  sequence for $G$}.  

\section{Technical Lemmas} 

We start with a bound on the complexity for the subshift obtained by forbidding the occurrences of some word: 
\begin{lemma}\label{lem:remove} 
Suppose $(X,\sigma)$ is a subshift and $w\in\mathcal{L}(X)$.  
If the cylinder set $[w]_0^+\cap X$ contains at least one aperiodic point, then for all $n\geq|w|$,  we have 
$$P_{X(w)}(n)\leq P_X(n)-(n-|w|+1).$$ 
\end{lemma} 
\begin{proof} 
Let $x\in[w]_0^+\cap X$ be aperiodic and let $n\geq|w|$ be fixed.  There are two cases to consider: 
\newline
\noindent
{\bf Case 1}:  Assume that every word of length $n$ that occurs in $x$ contains $w$ as a subword.  By the Morse-Hedlund Theorem and aperiodicity of $x$, there are at least $n+1$ distinct words of length $n$ that occur in $x$.  Since none of these words are in the language of $X(w)$, we have $P_{X(w)}(n)\leq P_X(n)-(n+1)$, which gives the inequality in the statement.  
\newline
\noindent
{\bf Case 2}: Assume there is a word of length $n$ that occurs in $x$ that does not contain $w$ as a subword.  Without loss of generality, we can assume that $x\in[w]_0^+$, and that either $(x_1,\dots,x_n)$ or $(x_{|w|-n-1},\dots,x_{|w|-2})$ does not contain $w$ as a subword (otherwise we replace $x$ with an appropriate shift of itself).  First suppose that $(x_1,\dots,x_n)$ does not contain $w$ as a subword and for each $y\in X$, let $\mathcal{W}(y)$ denote the unique $v\in\mathcal{L}_n(X)$ such that $y\in[v]_0^+$.  Then our assumption is that $w$ is the left-most subword (of length $|w|$) in $\mathcal{W}(x)$ and $w$ is not a subword of $\mathcal{W}(\sigma x)$.  Therefore each of the words $\mathcal{W}(x)$, $\mathcal{W}(\sigma^{-1}x)$, $\mathcal{W}(\sigma^{-2}x)$, \dots, $\mathcal{W}(\sigma^{|w|-n}x)$ contains $w$ as a subword, and the rightmost occurrence of $w$ as a subword of $\mathcal{W}(\sigma^{-i}x)$ begins at the $i^{th}$ letter, for each $0\leq i<n-|w|+1$.  It follows that these words are all distinct and none of them are words in the language of $X(w)$.  Therefore $P_{X(w)}(n)\leq P_X(n)-(n-|w|+1)$ in this case.  
On the other hand, if $w$ is not a subword of $(x_{|w|-n-1},\dots,x_{|w|-2})$, the argument is similar with the roles played by left and right reversed. 
\end{proof} 

We use this to bound the maximal length of a descending chain of subshifts obtained by forbidding a word at each step: 
\begin{lemma}\label{lem:chain} 
Suppose $(X,\sigma)$ is a subshift and $L\in\N$ is fixed.  Let 
$$ 
X=:X_0\supset X_1\supset X_2\supset\cdots\supset X_k 
$$ 
be a descending chain of nonempty subshifts, where for $0\leq i<k$ we have $X_{i+1}=X_i(w_i)$ for some 
word $w_i\in\mathcal{L}(X_i)$ with $|w_i|\leq L$.  Further suppose that for each $i=0,1,\dots,k-1$ the cylinder set $[w_i]_0^+\cap X_i$ contains at least one aperiodic point.  Then  
$$ 
k<P_X(2L-1)/L.  
$$  
\end{lemma} 
\begin{proof} 
By inductively applying Lemma~\ref{lem:remove} and using the fact that $|w_i|\leq L$ for all $i$, we have $P_{X_k}(2L-1)\leq P_X(2L-1)-kL$.  If $k\geq P_X(2L-1)/L$, then $P_{X_k}(2L-1)\leq0$.  However, this is impossible, as  nonempty subshift  has 
at least one word of every length.  
\end{proof} 

The proof of the following lemma is a straighforward modification of the proof of Lemma 3.1 in~\cite{CK}: 

\begin{lemma}\label{lem:extend} 
Let $d,N\in\N$ and suppose $(X,\sigma)$ is a subshift such that $P_X(n)\leq n^d$ for all $n\geq N$.  Define  $k_n$ to be 
$$
\min\{k\in\N\colon\text{no word $w\in\mathcal{L}_n(X)$ extends uniquely $k$ times to the right and left}\}. 
$$
Then there exists $C>0$ such that for all $n\geq N$, there exists $m\leq n\log n$ satisfying $k_m\geq Cn$.  Moreover, $C$ can be taken to be $log(2)/4d$. 
\end{lemma} 

\begin{proof} 
Suppose $P_X(n)\leq n^d$ for all $n\geq N$.  For contradiction, suppose that for all $C>0$ there exist arbitrarily large $n$ such that $k_m<Cn$ for all $m\leq n\log n$.  Then since every word of length $n$ can be extended in at least two distinct ways to a word of length $n+2k_n$ (by adding $k_n$ letters to each side), we have that 
$$ 
P_X(n+2k_n)\geq2P_X(n). 
$$ 
Therefore the assumption that $k_m<Cn$ for all $m\leq n\log n$ implies that 
$$P_X(\lfloor n\log n\rfloor)\geq2^{\lfloor n\log n/Cn\rfloor}=\Omega(n^{(\log 2)/C}).$$  
Fixing $0<C\leq(\log 2)/4d$, then for arbitrarily large $n$ we have 
$$ 
P_X(\lfloor n\log n\rfloor)\geq n^{(\log 2)/2C}\geq n^{2d}. 
$$ 
But this contradicts the fact that $P_X(\lfloor n\log n\rfloor)\leq(n\log n)^d<n^{2d}$ for all sufficiently large $n$. 
\end{proof}

Suppose $Y\subseteq X$ are two subshifts where $Y$ is obtained by forbidding a finite number of words from the language of $X$.  If $u\in\mathcal{L}(Y)$ and $u$ extends uniquely $T$ times to the right and left (as a word in the language of $Y$), it might not extend uniquely $T$ times to the right and left when thought of as a word in the language of $X$.  The following lemma resolves this issue, showing that if $u$ appears in some element of $X$ and is sufficiently far from any occurrence of the forbidden words, then $u$ behaves as though it occurs in the language of $Y$: 
\begin{lemma}\label{lem:extend2} 
Let $(X,\sigma)$ be a subshift and let $w_1,\dots,w_{k-1}\in\mathcal{L}(X)$.  Suppose 
$$ 
Y=\{x\in X\colon\sigma^i(x)\notin[w_j]_0^+\text{ for any }i\in\Z\text{ and }j\in\{1,2,\dots,k\}\} 
$$ 
and suppose $u\in\mathcal{L}(Y)$ is a word which extends uniquely (in $\mathcal{L}(Y)$) at least $T$ many times to the right and left.  Let $v\in\mathcal{L}_{|u|+2T}(Y)$ be the unique word such that $u$ is obtained by removing the rightmost and leftmost $T$ letters from $v$.  Then there exists $D\in\N$ such that for any $x\in[u]_0^+\cap X$, if $\sigma^ix\notin[w_j]_0^+$ for any $-D\leq i\leq D$ and $1\leq j<k$, then $\sigma^{-T}x\in[v]_0^+$. 
\end{lemma} 

\begin{proof} 
For contradiction, suppose not.  For each $D\in\N$, choose $x_D\in[u]_0^+\cap X$ such that $\sigma^{-T}x_D\notin[v]_0^+$ and $\sigma^ix_D\notin[w_j]_0^+$ for any $-D\leq i\leq D$ and $1\leq j<k$.  Let $x\in X$ be a limit point of $\{x_D\colon D\in\N\}$.  Then $x\in[u]_0^+$ and $\sigma^ix\notin[w_j]_0^+$ for any $i\in\Z$ and $1\leq j<k$.  
Thus $x\in Y$, and since $u$ extends uniquely $T$ times to the right and left (as a subword of $Y$) we have $\sigma^{-T}x\in[v]_0^+$.  This contradicts the fact that $\sigma^{-T}x_D\notin[v]_0^+$ for all $D\in\N$ and $x$ is a limit point of $\{x_D\colon D\in\N\}$. 
\end{proof} 


We now generalize Lemma~\ref{lem:useful} to a form which is more useful in our setting. 

\begin{lemma}\label{lem:virtualextend} 
Let $R\in\N$ be fixed and suppose $(X,\sigma)$ is a subshift.  Let 
$$ 
X=:X_0\supset X_1\supset X_2\supset\cdots\supset X_k 
$$ 
be a descending chain of nonempty subshifts, where for $0\leq i<k$ there exists $w_i\in\mathcal{L}(X_i)$ such that $w_i$ extends uniquely at least $2R$ times to the right and left (as a word in $\mathcal{L}(X_i)$), and $X_{i+1}=X_i(w_i)$.  Suppose further that there exists $w_k\in\mathcal{L}(X_k)$ that extends at least $2R$ times to the right and left but for which $X_k(w_k)=\varnothing$.  Let $\tilde{w_i}\in\mathcal{L}(X_i)$ be the unique word of length $|w_i|+4R$ obtained by extending $w_i$ by $2R$ letters on each side.  If $\varphi,\psi\in\Aut_R(X)$ are such that $\varphi(\tilde{w}_i)=\psi(\tilde{w}_i)$ for all $i=0,1,\dots,k$, then there exists $D$ such that for all $i=0,1,2,\dots,k$ the automorphism $(\varphi^{-1}\circ\psi)$ preserves all occurrences of $\tilde{w}_0$ and preserves all occurrences of $\tilde{w}_i$ that occur at least $D$ units from $\tilde{w}_0,\tilde{w}_1,\dots,\tilde{w}_{i-1}$. 
\end{lemma} 
\begin{proof} 
First we show that $\varphi^{-1}\circ\psi$ preserves occurrences of $\tilde{w}_0$.  Let $x\in[\tilde{w}_0]_0^+\cap X$.  
Since $\psi(\tilde{w}_0)=\varphi(\tilde{w}_0)$, it follows that $\psi(x)\in[\varphi(\tilde{w}_0)]_0^+\cap X$.  Note that $\varphi(\tilde{w}_0)$ is $2R$ letters shorter than $\tilde{w}_0$, since the block code defining $\varphi$ has range $R$.  
Since $\varphi^{-1}([\varphi(\tilde{w}_0)]_0^+\cap X)\subset[w_0]_0^+\cap X$, 
we have that $(\varphi^{-1}\circ\psi)(x)\in[w_0]_0^+\cap X$. 
   But $[w_0]_0^+\cap X=[\tilde{w}_0]_0^+\cap X$ since $w_0$ extends uniquely $2R$ times to the right and left.  Since $x\in[\tilde{w}_0]_0^+\cap X$ was arbitrary, we have $(\varphi^{-1}\circ\psi)([\tilde{w}_0]_0^+\cap X)\subset[\tilde{w}_0]_0^+\cap X$.  Therefore $\varphi^{-1}\circ\psi$ preserves occurrences of $\tilde{w}_0$. 
 Clearly the roles of $\varphi$ and $\psi$ can be interchanged, 
 and so it also follows that 
 $\psi^{-1}\circ\varphi=(\varphi^{-1}\circ\psi)^{-1}$ preserves occurrences of $\tilde{w}_0$. 

For the second statement, we proceed by induction.  Assume 
that we have shown that $\varphi^{-1}\circ\psi$ and $\psi^{-1}\circ\varphi$ preserve occurrences of $\tilde{w}_0$ and there exists $D_k$ such that for all $i<j$, 
the automorphisms $\varphi^{-1}\circ\psi$ and $\psi^{-1}\circ\varphi$ preserve occurrences of $\tilde{w}_i$ 
that occur at least $D_k$ units from $\tilde{w}_0,\tilde{w}_1,\dots,\tilde{w}_{i-1}$.  We  
show that there exists $D_{k+1}\geq D_k$ such that $\varphi^{-1}\circ\psi$ and $\psi^{-1}\circ\varphi$ also preserve occurrences of $\tilde{w}_k$ that occur 
at least $D_{k+1}$ units from $\tilde{w}_0,\dots,\tilde{w}_{k-1}$.  By Lemma~\ref{lem:extend2} applied to the subshifts $X$ and $Y:=X_k$, there exists 
$\tilde{D}_{k+1}\geq D_k$ such that for any $x\in[\tilde{w}_k]_0^+\cap X$, if $\sigma^ix\notin[\tilde{w}_j]_0^+$ for any $-\tilde{D}_{k+1}\leq i\leq\tilde{D}_{k+1}$ and $1\leq j<k$, then $\sigma^{-2R}x\in[\tilde{w}_k]_0^+$.  Define $D_{k+1}:=(k+1)\cdot\tilde{D}_{k+1}$ and let $x\in[\tilde{w}_k]_0^+\cap X$ be such that $\sigma^ix\notin[\tilde{w}_j]_0^+$ for any $-D_{k+1}\leq i\leq D_{k+1}$ and $1\leq j<k$.  Define 
$$ 
\mathcal{P}_0:=\{p\in\Z\colon\sigma^px\in[\tilde{w}_0]_0^+\}. 
$$ 
Since $\varphi^{-1}\circ\psi$ and $\psi^{-1}\circ\varphi$ both preserve occurrences of $\tilde{w}_0$, observe that $\mathcal{P}_0$ is equal to the set $\{p\in\Z\colon\sigma^p(\varphi^{-1}\circ\psi)x\in[\tilde{w}_0]_0^+\}$ (in other words, occurrences of $\tilde{w}_0$ can neither be created nor destroyed by applying $\varphi^{-1}\circ\psi$ to $x$).  Next define 
$$ 
\mathcal{P}_1:=\{p\in\Z\colon\sigma^px\in[\tilde{w}_1]_0^+\}. 
$$ 
Since $\varphi^{-1}\circ\psi$ and $\psi^{-1}\circ\varphi$ both preserve occurrences of $\tilde{w}_1$ when they 
occur at least $D_k$ units from $\tilde{w}_0$, then for any $t\in\N$ any element of 
$$ 
\mathcal{P}_1\triangle\{p\in\Z\colon\sigma^p(\varphi^{-1}\circ\psi)^t(x)\in[\tilde{w}_1]_0^+\} 
$$ 
is within distance $D_k$ of an element of $\mathcal{P}_0$.  Further defining for each $1<i<k$ the set 
$$ 
\mathcal{P}_i:=\{p\in\Z\colon\sigma^px\in[\tilde{w}_i]_0^+\}, 
$$ 
it follows by induction that for any $t\in\N$, any element of 
$$ 
\mathcal{P}_i\triangle\{p\in\Z\colon\sigma^p(\varphi^{-1}\circ\psi)^t(x)\in[\tilde{w}_i]_0^+\} 
$$ 
lies either within distance $iD_k$ of an element of $\mathcal{P}_0$, within distance $(i-1)D_k$ of an element of $\mathcal{P}_1$, \dots, or within distance $D_k$ of an element of $\mathcal{P}_{k-1}$. 
Recall that, by assumption, the set 
$$ 
\left\{p\in\Z\colon\sigma^px\in\bigcup_{i=0}^{k-1}[\tilde{w}_i]_0^+\right\} 
$$ 
does not contain any element within distance $D_{k+1}$ of the origin.  But $\tilde{D}_{k+1}<D_{k+1}-k\cdot D_k$, 
and so for any $t\in\N$, the set 
$$ 
\left\{p\in\Z\colon\sigma^p(\varphi^{-1}\circ\psi)^t(x)\in\bigcup_{i=0}^{k-1}[\tilde{w}_i]_0^+\right\}
$$ 
does not contain any element within distance $\tilde{D}_{k+1}$ of the origin.  
However $x\in[\tilde{w}_k]_0^+$ and so, as previously, we have $(\varphi^{-1}\circ\psi)x, (\psi^{-1}\circ\varphi)x\in[w_k]_0^+$.  Since this occurrence of $w_k$ (in the element $x$) is at least $\tilde{D}_{k+1}$ units from any occurrence of $\tilde{w}_0,\dots,\tilde{w}_{k-1}$, we have $(\varphi^{-1}\circ\psi)x, (\psi^{-1}\circ\varphi)x\in[\tilde{w}_k]_0^+$. 
\end{proof} 

The following lemma allows us to adapt techniques from~\cite{CK} which relied on the fact that in a minimal shift all words occur syndetically. 

\begin{lemma}
\label{lem:syndetic} 
Suppose $(X,\sigma)$ is a subshift.  Let 
$$ 
X=:X_0\supset X_1\supset X_2\supset\cdots\supset X_k 
$$ 
be a descending chain of nonempty subshifts where for each $0\leq i<k$ there exists $w_i\in\mathcal{L}(X_i)$ such that $w_i$ extends uniquely at least $T$ times to the right and left (as a word in $\mathcal{L}(X_i)$), and $X_{i+1}=X_i(w_i)$.  Suppose further that there exists $w_k\in\mathcal{L}(X_k)$ that extends at least $T$ times to the right and left but for which $X_k(w_k)=\varnothing$.  Let $\tilde{w_i}\in\mathcal{L}(X_i)$ be the unique word of length $|w_i|+2T$ obtained by extending the word $w_i$ by 
$T$ letters on each side.  Finally, for each $i=1,2,\dots,k$, let $D\in\N$ be the constant obtained from Lemma~\ref{lem:virtualextend}.  Then there exists $G\in\N$ such that for any $x\in X$, the set 
\begin{multline*} 
\mathcal{S}_x:=\{j\in\Z\colon\sigma^ix\in[\tilde{w}_0]_0^+\}\cup\{j\in\Z\colon\sigma^jx\in[\tilde{w}_i]_0^+\text{ for some }1\leq i\leq k \\ \text{ and }\sigma^sx\notin[\tilde{w}_t]_0^+\text{ for any }t<i\text{ and any }j-D\leq s\leq j+D+|\tilde{w}_i|-1\} 
\end{multline*}
is syndetic with gap at most $G$. 
\end{lemma} 
\begin{proof} 
If not, then for each $G\in\N$ there exists $x_G\in X$ such that $\mathcal{S}_x$ is not syndetic with gap less than $G$.  Without loss (shifting $x$ if necessary), we can assume that $\{-\lfloor G/2\rfloor,\dots,0,\dots,\lfloor G\rfloor/2\}\cap\mathcal{S}_x=\varnothing$.  Since $X$ is compact, we can pass to a subsequence of $(x_G)_{G\in\N}$ converging to some 
$y\in X$.  Then $\mathcal{S}_y=\varnothing$.  Therefore $y\in X_k(w_k)$, a contradiction of the assumption that $X_k(w_k)=\varnothing$. 
\end{proof}

We use this to describe the set of automorphisms preserving occurrences of the sequence of words:  
\begin{lemma}\label{lem:subgroup} 
Assume $(X,\sigma)$ is a subshift and $D\in\N$.  
Let $\mathcal{W}=\{\tilde{w}_0,\tilde{w}_1,\dots,\tilde{w}_k\}\subset\mathcal{L}(X)$ be a finite set of words for which there exists $G\in\N$ such that for any $x\in X$, the set 
\begin{multline*}
\mathcal{S}_x:=\{j\in\Z\colon\sigma^jx\in[\tilde{w}_0]_0^+\}\cup\{j\in\Z\colon\sigma^jx\in[\tilde{w}_i]_0^+\text{ for some }1\leq i\leq k\\ 
\text{ but }\sigma^sx\notin[\tilde{w}_t]_0^+\text{ for any }t<i\text{ and any }j-D\leq s\leq j+D+|\tilde{w}_i|-1\} 
\end{multline*}
is syndetic with gap at most $G$.  Let $R\in\N$ be such that $|\text{w}_i|>2R$ for all $\tilde{w}_i\in\mathcal{W}$ and define 
\begin{multline*}
\mathcal{H}:=\{\varphi\in\Aut_R(X)\colon\varphi\text{ preserves occurrences of }\tilde{w}_0\text{ and occurrences of  } \\ \tilde{w}_i\text{ that occur at least $D$ units from }\tilde{w}_0,\tilde{w}_1,\dots,\tilde{w}_{i-1}\text{ for all }1\leq i\leq k\}. 
\end{multline*}
Then $\langle\mathcal{H}\rangle$ is finite. 
\end{lemma} 
\begin{proof} 
Suppose $w_1,w_2\in\mathcal{W}$ and $u\in\mathcal{L}(X)$ is such that $[w_1uw_2]_0^+\cap X\neq\varnothing$.  Let $\varphi\in\mathcal{H}$.  Then $\varphi$ has range $R$ and by assumption we have $|w_1|>2R$ and $|w_2|>2R$.  Therefore there exists $v\in\mathcal{L}_{|u|}(X)$ such that $\varphi([w_1uw_2]_0^+\cap X)\subset[w_1vw_2]_0^+$, 
since $\varphi$ preserves occurrences of $w_1, w_2$ and the block code defining $\varphi$ does not have access to any information to the left of $w_1$ or to the right of $w_2$ while it acts on the word $u$.  As any element of $\langle\mathcal{H}\rangle$ can be written as a product of elements of $\mathcal{H}$, 
it follows that if $\alpha\in\langle\mathcal{H}\rangle$ then this same property holds: there exists $v\in\mathcal{L}_{|u|}(X)$ such that $\alpha([w_1uw_2]_0^+\cap X)\subset[w_1vw_2]_0^+\cap X$.  

By Lemma~\ref{lem:syndetic}, any element $x\in X$ can be (non-canonically) decomposed as 
$$ 
x=\cdots w_{-2}u_{-2}w_{-1}u_{-1}w_0u_0w_1u_1w_2u_2\cdots, 
$$ 
where $w_i\in\mathcal{W}$ for all $i\in\Z$ and $u_i\in\mathcal{L}(X)$ satisfies $|u_i|\leq G+\max\{|w|\colon w\in\mathcal{W}\}$ 
for all $i\in\Z$.  Therefore if $\alpha\in\langle\mathcal{H}\rangle$, then $\alpha$ is determined entirely by its action on sets of the form $[w_1uw_2]_0^+\cap X$ where $|u|\leq G+\max\{|w|\colon w\in\mathcal{W}\}$.  As there are only finitely many such sets and the image of each of these is another set of the same form, there are only finitely many elements of $\langle\mathcal{H}\rangle$. 
\end{proof} 

Our final technical lemma quantifies a property of functions which grow subexponentially. 

\begin{lemma}\label{lem:subexp} 
Let $g\colon\N\to\N$ be such that $\log g(n)=o(n)$.  For any $k\in\N$ and all sufficiently small $\varepsilon>0$, 
there exists $M\in\N$ such that if $N\geq M$ and if $f\colon\{0,1,\dots,N\}\to\N$ is a nondecreasing function satisfying 
$f(N)\leq g(N)$, then there exists $x\in\{0,1,\dots,N-k\}$ such that 
$$ 
\frac{f(x+k)-f(x)}{f(x)}<\varepsilon. 
$$ 
\end{lemma} 
\begin{proof} 
Let $k\in\N$ and $\varepsilon>0$ be fixed.  Find $M\in\N$ such that for all $N\geq M$ we have 
\begin{equation}\label{eq:est} 
\frac{\log g(N)}{N}<\varepsilon/4k. 
\end{equation} 
Without loss of generality, 
we can assume that $M>k$.  Let $N\geq M$ and let $f\colon\{0,1,\dots,N\}\to\N$ be nondecreasing.  Suppose that for all $0\leq x\leq N-k$ we have 
$$ 
\frac{f(x+k)-f(x)}{f(x)}\geq\varepsilon. 
$$ 
Then by induction, $f(nk)\geq(1+\varepsilon)^nf(0)$ for all $0\leq n<\lfloor N/k\rfloor$.  In particular, since $f$ is nondecreasing, 
$$ 
f(N)\geq(1+\varepsilon)^nf(0) 
$$ 
where $n=\lfloor N/k\rfloor$.  Therefore 
$$ 
\log g(N)\geq\log f(N)\geq n\log(1+\varepsilon)+\log f(0)\geq\frac{N}{2k}\log(1+\varepsilon)>\frac{N\varepsilon}{3k} 
$$ 
for all sufficiently small $\varepsilon$, a contradiction of~\eqref{eq:est}. 
\end{proof} 

\section{Amenability of $\Aut(X)$} 

Our goal in this section is to prove Theorem~\ref{th:main}.  We do this first with an added assumptions: that $(X,\sigma)$ has dense aperiodic points. 

\begin{theorem}\label{thm:dense} 
Let $X$ be a subshift with dense aperiodic points and suppose $P_X(n)=o(n^2/\log^2 n)$.  Then $\Aut(X)$ is amenable (as a countable discrete group). 
\end{theorem} 

\begin{proof} 
Fix $R\in\N$ and let $C$ be as in Lemma~\ref{lem:extend}.  Define $X_0=X$ and by Lemma~\ref{lem:extend}, 
choose a word $w_0\in\mathcal{L}(X_0)$ such that $|w_0|\leq 2R/C\log 2R/C$ and $w_0$ extends uniquely at least $2R$ 
times to the right and to the left.  
Define $\tilde{w}_0$ to be the (unique) extension of $w_0$ exactly $2R$ times to each side and set $X_1:=X_0(\tilde{w}_0)$.  
Continue this process inductively:  once we have constructed the nonempty subshift $X_i$, 
apply Lemma~\ref{lem:extend} to find a word $w_i\in\mathcal{L}(X_i)$ such that $|w_i|\leq 2R/C\log 2R/C$ and such that $w_i$ extends at least $2R$ times to the right and to the left.  
Define $\tilde{w}_i$ to be the (unique) extension (in $\mathcal{L}(X_i)$) of $w_i$ exactly $2R$ times to each side and set $X_{i+1}:=X_i(\tilde{w}_i)$.  If $X_{i+1}$ is empty, the process ends.  By Lemma~\ref{lem:chain} this process ends after at most 
$k_R$ steps, where $k_R\leq P_X(2L-1)/L$ and $L=\lfloor 2R/C\log 2R/C\rfloor$.  It follows that, as a function of $R$, $k_R=o(R/\log R)$.  To summarize, for any $R\in\N$ we have constructed a sequence of nonempty subshifts 
$$ 
X:=X_0\supset X_1\supset X_2\supset\cdots\supset X_{k_R} 
$$ 
such that for each $i=0,1,\dots,k_R-1$ there exists $w_i\in\mathcal{L}(X_i)$ 
that extends uniquely (in $\mathcal{L}(X_i)$) at least $2R$ times to each side and is such that 
$|w_i|\leq 2R/C\log 2R/C$ and $X_{i+1}=X_i(\tilde{w}_i)$ (where $\tilde{w}_i$ is the extended version of $w_i$).  Note that for all sufficiently large $R$, $|\tilde{w}_i|\leq R^2$.  For fixed $R\in\N$, let $\mathcal{W}_R=\{\tilde{w}_1,\dots,\tilde{w}_{k_R}\}$ and let $\mathcal{G}_R\subset\Aut(X)$ be the subgroup of automorphisms generated by 
$$ 
\{\varphi\in\Aut_R(X)\colon \varphi([\tilde{w}]_0^+\cap X)\subset[\tilde{w}]_0^+\cap X \text{ for all }\tilde{w}\in\mathcal{W}_R\}. 
$$ 
By Lemma~\ref{lem:subgroup}, $\mathcal{G}_R$ is finite.  Furthermore, by the Pigeonhole Principle, 
if $\mathcal{S}\subset\Aut_R(X)$ is any set satisfying 
$$
|\mathcal{S}|>P_X(R^2)^{k_R}, 
$$ 
then there exist $\varphi,\psi\in\mathcal{S}$ such that $\varphi(\tilde{w}_i)=\psi(\tilde{w}_i)$ for all $i=0,1,2,\dots,k_R$. In other words: 
if $\mathcal{S}\subset\Aut_R(X)$ is any set satisfying 
$
|\mathcal{S}|>P_X(R^2)^{k_R}, 
$ 
then there exist $\varphi,\psi\in\mathcal{S}$ such that $(\varphi^{-1}\circ\psi)\in\mathcal{G}_R$. 

Since $P_X(R^2)\leq R^4/\log^2(R^2)\leq R^4$ for all sufficiently large $R$, we have 
$$P_X(R^2)^{k_R}\leq(R^4)^{o(R/\log(R))}=e^{o(R)},$$ 
meaning that this grows subexponentially in $R$.  Define $g\colon\N\to\N$ by $g(R):=P_X(R^2)^{k_R}$.  
Then by Lemma~\ref{lem:subexp}, for any $k\in\N$ and any sufficiently small $\varepsilon>0$, there exists $M$ such for any $N\geq M$ and any nondecreasing function $f\colon\{0,1,\dots,N\}\to\N$, which satisfies $f(N)\leq g(N)$, there exists $0\leq x\leq N-k$ such that 
$$ 
\frac{f(x+k)-f(x)}{f(x)}<\varepsilon. 
$$ 

We are now ready to prove that $\Aut(X)$ is amenable.  Let $k\in\N$ be fixed.  Choose $\varepsilon<1/k$ 
 sufficiently small such that Lemma~\ref{lem:subexp} applies 
 and let $M$ be the constant obtained from this lemma.  
 Choose $R>\max\{k,M\}$ large enough such that 
$$ 
\frac{g(R+k)-g(R)}{g(R)}<\frac{\varepsilon}{4k}. 
$$ 
Let $f\colon\{0,1,2\dots,R\}\to\N$ be the function 
\begin{multline*}
f(n):=|\{(\varphi(\tilde{w}_0),\varphi(\tilde{w}_1),\dots,\tilde{w}_{k_R})\colon\varphi\in\Aut_n(X)\hookrightarrow\Aut_R(X)\\ \text{ and }\mathcal{W}_R=\{\tilde{w}_0,\dots,\tilde{w}_{k_R}\}\}|. 
\end{multline*}
 Here, for $n\leq R$, the notation $\Aut_n(X)\hookrightarrow\Aut_R(X)$ refers to the embedded image of $\Aut_n(X)$ in $\Aut_R(X)$ obtained by using the natural identification of a range $n$ block code as a range $R$ block code.  It follow that if $\mathcal{S}\subset\Aut_n(X)$ is any set containing more than $f(n)$ elements, then there exist $\varphi,\psi\in\Aut_n(X)$ such that $(\varphi^{-1}\circ\psi)\in\mathcal{G}_R$.  In other words, $\Aut_n(X)$ can be covered by $f(n)$ many cosets of $\mathcal{G}_R$. 

By Lemma~\ref{lem:subexp}, there exists $0\leq n\leq R-k$ such that 
$$ 
\frac{f(n+k)-f(n)}{f(n)}<\varepsilon/2. 
$$ 
Fix such $n$ and let $\varphi_1,\dots,\varphi_{f(n)}\in\Aut_n(X)$ be representatives of distinct cosets of $\mathcal{G}_R$ and such that $\Aut_n(X)$ is contained in 
$$ 
F_k:=\bigcup_{i=0}^{f(n)}\varphi_i\cdot\mathcal{G}_R. 
$$ 
Observe that $F_k$ is finite and contains $\Aut_n(X)$.  Now let $\varphi_{f(n)+1},\dots,\varphi_{f(n+k)}\in\Aut_{n+k}(X)$ be $f(n+k)-f(n)$ additional representatives of distinct cosets of $\mathcal{G}_R$ and such that $\Aut_{n+k}(X)$ is contained in 
$$ 
\tilde{F}_k:=\bigcup_{i=0}^{f(n+k)}\varphi_i\cdot\mathcal{G}_R. 
$$ 
Observe that if $\psi\in\Aut_k(X)$ then for any $i=0,1,\dots,f(n)$ we have $(\psi\circ\varphi_i)\in\Aut_{n+k}(X)$.  Therefore 
$$ 
\psi\cdot F_k\subset\tilde{F}_k
$$ 
and since $F_k\subset\tilde{F}_k$, we have 
$$ 
\frac{|F_k\triangle(\psi\cdot F_k)|}{|F_k|}\leq\frac{2|\tilde{F}_k\setminus F_k|}{|F_k|}=\frac{2(f(n+k)-f(n))}{f(n)}<\varepsilon\leq\frac{1}{k}. 
$$ 
Let $n_k\in\N$ be the constant $n$ constructed above.  Observe that $n_k\to\infty$ as $k\to\infty$. 

Construct the set $F_k$ for each $k\in\N$.  We claim that $(F_k)_{k\in\N}$ is a F{\o}lner sequence in $\Aut(X)$.  By construction, $F_k$ is finite for each $k$ and we have shown that 
$$ 
\frac{|F_k\triangle(\psi\cdot F_k)|}{|F_k|}<\frac{1}{k} 
$$ 
for each $k\in\N$.  Finally, since $n_k\to\infty$ as $k\to\infty$, we have that if $\psi\in\Aut(X)$ then $\{k\colon\psi\notin F_k\}$ is finite.  
Thus we have constructed a F{\o}lner sequence for $\Aut(X)$ and so it is amenable. 
\end{proof} 
 
 We use this to complete the proof of Theorem~\ref{th:main}: 
\begin{proof}[Proof of Theorem~\ref{th:main}]
Let $Y\subset X$ be the closure of the aperiodic points in $X$.  By Theorem~\ref{thm:dense}, $\Aut(Y)$ is amenable.  For any $\varphi\in\Aut(X)$, observe that $x\in X$ is aperiodic if and only if $\varphi(x)$ is aperiodic.  Therefore  for any $\varphi\in\Aut(X)$, 
we have $\varphi(Y)=Y$ and the map $h\colon\Aut(X)\to\Aut(Y)$ given by $h(\varphi):=\varphi|_Y$ is a homomorphism.  Since the image of $h$ is a closed subgroup of $\Aut(Y)$, it is amenable.  
Thus to check that $\Aut(X)$ is amenable, it suffices to check that $\ker(h)$ is amenable. 

To show this, 
it suffices to show that any finitely generated subgroup of $\ker(h)$ is amenable.  Let $\varphi_1,\dots,\varphi_m\in\ker(h)$.  We claim that the set
$$ 
\mathcal{S}:=\{x\in X\colon\varphi_i(x)\neq x \text{ for at least one element of } \langle\varphi_1,\dots,\varphi_m\rangle\} 
$$ 
is finite.   For contradiction, suppose $\mathcal{S}$ is infinite.  Choose $R\in\N$ such that $\varphi_1,\dots,\varphi_m\in\Aut_R(X)$.  By construction, if $w\in\mathcal{L}(Y)$ is a word of length $R$ and if $\id\in\Aut(Y)$ denotes the identity, 
then $\varphi_i(w)=\id(w)$ (as a block map).  Therefore if $x\in\mathcal{S}$,  there exists $j\in\Z$ such that $\sigma^jx\in[u]_0^+$ for some word $u\in\mathcal{L}(X)\setminus\mathcal{L}(Y)$ of length $R$ (otherwise $x$ is comprised entirely of words of length $R$ on which $\varphi$ acts as the identity).  For each $x\in\mathcal{S}$ choose a word $u_x\in\mathcal{L}_R(X)\setminus\mathcal{L}_R(Y)$ such that $\sigma^jx\in[u_x]_0^+$ for some $j\in\Z$.  Since $\mathcal{L}_R(X)\setminus\mathcal{L}_R(Y)$ is finite, there exists $u\in\mathcal{L}_R(X)\setminus\mathcal{L}_R(Y)$ such that $u_x=u$ for infinitely many $x\in\mathcal{S}$.  For each such $x$, let $j_x\in\Z$ be such that $\sigma^{j_x}x\in[u]_0^+$.  Every infinite collection of points has an aperiodic limit point, and so there is some aperiodic $y\in[u]_0^+$.  This contradicts the fact that $u\notin\mathcal{L}(Y)$.  Therefore $\mathcal{S}$ is finite, 
proving the claim. 

Since the set $\mathcal S$ is finite, it follows that $\langle\varphi_1,\dots,\varphi_m\rangle$ is finite (and hence amenable).
\end{proof} 

\begin{corollary} 
Let $X$ be a subshift and suppose there exists $\varepsilon>0$ such that $P_X(n)=O(n^{2-\varepsilon})$.  Then $\Aut(X)$ is amenable (as a countable discrete group). 
\end{corollary} 
\begin{proof} 
Any function which is $O(n^{2-\varepsilon})$ is also $o(n^2/\log^2 n)$, and so this 
follows immediately from Theorem~\ref{th:main}. 
\end{proof} 

\section{Shifts of subquadratic growth} 

Recall that if $G$ is a group and $g_1,g_2\in G$ then 
$$ 
\langle g_1,g_2\rangle^+:=\{g_1^ig_2^j\colon i,j\in\N\cup\{0\}\} 
$$ 
is the {\em semigroup generated by $g_1$ and $g_2$} in $G$.  This semigroup is {\em free} if whenever $i_1,i_2,j_1,j_2\in\N\cup\{0\}$ and $(i_1,j_1)\neq(i_2,j_2)$, we have $g_1^{i_1}g_2^{j_1}\neq g_1^{i_2}g_2^{j_2}$.  If this semigroup is free, it is said to have {\em rank $2$} because it is generated by two elements of $G$. 

Let $(X,\sigma)$ be a subshift.  The {\em full group} of $(X,\sigma)$, denoted $[\sigma]$, is the group of all maps $\varphi\colon X\to X$ such that there exists $k\colon X\to\Z$ such that $\varphi(x)=\sigma^{k(x)}(x)$ for all $x\in X$.  The group $[\sigma]\cap\Aut(X)$ is the group of all {\em orbit-preserving} automorphisms.  This is a normal, abelian subgroup of $\Aut(X)$. 

We recall the statement of Theorem~\ref{th:main2}:

\newtheorem*{thm2*}{Theorem~\ref{th:main2}} 
\begin{thm2*}
{\em Let $(X,\sigma)$ be a subshift such that $\liminf P_X(n)/n^2=0$.  Then $\Aut(X)$ does not contain a free semigroup of rank $2$.} 
\end{thm2*}

\noindent Our main tool to prove this theorem is the following rephrasing given in~\cite{CK3} of a result of Quas and Zamboni: 
\begin{lemma}[Quas-Zamboni~\cite{QZ}]\label{thm:qz} 
Let $n,k\in\N$.  Then there exists a finite set $F\subset\Z^2\setminus\{(0,0)\}$ (which depends on $n$ and $k$) such that for every $\eta\in\A^{\Z^2}$ satisfying $P_{\eta}(n,k)\leq nk/16$ there exists a vector $v\in F$ such that $\eta(x+v)=\eta(x)$ for all $x\in\Z^2$. 
\end{lemma} 
We now adapt the technique developed in~\cite{CK3} to prove Theorem~\ref{th:main2}. 

\begin{proof}[Proof of Theorem~\ref{th:main2}] 
Let $X\subset\A^{\Z}$ be a subshift and suppose $P_X(n)=o(n^2)$.  For each $\varphi\in\Aut(X)$ and each $x\in X$, define $\eta_{\varphi,x}\in\A^{\Z^2}$ by setting $\eta_{\varphi,x}(i,j):=(\varphi^j\sigma^i)(x)$ (this is the space time of the system).  Finally let 
$$
Y_{\varphi,x}\subset\A^{\Z^2}:=\overline{\{\eta_{\varphi,x}\circ S^jT^i\colon(i,j)\in\Z^2\}} 
$$
where $S,T\colon\Z^2\to\Z^2$ are the vertical and horizontal shifts, respectively: $S(i,j):=(i,j+1)$ and $T(i,j):=(i+1,j)$.  Since $P_X(n)=o(n^2)$, it follow from~\cite[Lemma 2.1]{CK3} that $P_{Y_{\varphi,x}}(n,n)=o(n^2)$ (however the rate at which $P_X(n)/n^2$ tends to zero depends on the range of $\varphi$). 

For contradiction, suppose $\varphi,\psi\in\Aut(X)$ generate a free semigroup.  Let $Z\subset X$ be the closure of the aperiodic elements of $X$ and as already noted, $Z$ is $\Aut(X)$-invariant.  By Theorem~\ref{thm:qz}, there exists a finite set $F\subset\Z^2\setminus\{(0,0)\}$ such that for any $x\in Z$ the maps $\eta_{\varphi,x}$ and $\eta_{\psi,x}$ are both periodic with some period vector in $F$.  Note that if $x$ is aperiodic, then $\eta_{\varphi,x}$ and $\eta_{\psi,x}$ cannot be horizontally periodic.  
Therefore there exists $M\in\N$ such that $\eta_{\varphi,x}$ and $\eta_{\psi,x}$ both have period vectors with $y$-coordinate $M$, 
for all aperiodic $x\in Z$.  It follows that $\varphi^Mx$ and $\psi^Mx$ are both shifts of $x$.  As this holds for all aperiodic $x\in X$, 
the restrictions of $\varphi^M$ and $\psi^M$ to $Z$ are both elements of the abelian group $\Aut(Z)\cap[\sigma]$.  
In particular, the restriction of the commutator $[\varphi^{M},\psi^M]$
 to $Z$ is the identity.  
 
 Suppose the range of $[\varphi^{M},\psi^M]$ is $R$.  
 If $w\in\mathcal{L}_R(X)$ is such that $[w]_0^+$ contains an aperiodic point, then $[\varphi^{M},\psi^M]$ 
 acts like the identity map (when thought of as a range $R$ block code) on $w$.  It follows that if $[\varphi^{M},\psi^M]$
  does not act like the identity map on $[w]_0^+$, then $[w]_0^+\cap X$ 
  does not contain any aperiodic points.  Furthermore, this means  that   $[w]_0^+\cap X$ 
  can not contain periodic points of arbitrarily large period.  
Therefore, $X\setminus Z$ is finite and so there exists $k\in\N$ such that 
  $[\varphi^{kM},\psi^{kM}]$ is the identity on $X$; contradicting the fact that $\varphi^{kM}$ 
  and $\psi^{kM}$ do not commute (since $\varphi$ and $\psi$ generate a free semigroup). 
\end{proof} 

\noindent While Theorem~\ref{th:main2} applies to a larger class of shifts than Theorem~\ref{th:main}, it does not conclude that $\Aut(X)$ amenable.  Nevertheless, Theorem~\ref{th:main2} does give algebraic information about $\Aut(X)$ (in particualr it cannot contain a non-abelian free subgroup) and a recent result of Salo and Schraudner shows that it is essentially optimal: 

\begin{theorem}[Salo-Schraudner~\cite{SaSc}] 
There exists a subshift $(X,\sigma)$ such that 
$$ 
P_X(n)=(n+1)^2 
$$ 
and is such that $\Aut(X)$ contains a free semigroup of rank $2$. 
\end{theorem} 
The example that they construct is the Cartesian product of two copies of the subshift $X$ on the alphabet $\{0,1\}$ where each $x\in X$ contains at most one occurrence of $1$. 

We present here a second example (different from that of~\cite{SaSc}) of a shift of quadratic growth whose automorphism group contains a free semigroup of rank $2$. 

Let $\mathcal{A}$ be the eight letter alphabet $\{0,1,a,b,p,1p,ap,bp\}$.  We consider the subshift $X\subset\A^{\Z}$ consisting of the following: 
	\begin{enumerate} 
	\item the coloring of all $0$'s; 
	\item any coloring which is all $0$'s except at a single location where it is one of $1$, $a$, $b$, $p$, $1p$, $ap$, or $bp$; 
	\item any coloring which is all $0$'s except at two locations, one of which is $p$ and the other of which is one of $1$, $a$, or 
	$b$. 
	\end{enumerate} 
We leave it to the reader to check that this does indeed form a (closed) subshift and that its complexity function grows quadratically.  
To show that this contains a free semigroup, we 
define two automorphisms of $X$ which we call $\varphi_a$ and $\varphi_b$.  These are 
range $1$ block codes and we claim they generate a free semi-group of rank $2$.  Rather than define them on each block, we find the following description helpful.  A person, $p$, walks down a bi-infinite hallway.  At one location in the hallway there is a nail on the wall and there are two possible pictures, $a$ and $b$, than can be hung on the nail.  When the nail is unoccupied, its location is denoted $1$.  When it holds picture $a$, its location is $a$.  When it holds picture $b$, its location is $b$.  If the person is standing in front of the nail, the person/nail is denoted $1p$, $ap$, or $bp$ (depending on the state of the nail).  Now we can define our automorphisms.  When $\varphi_a$ is applied to an element of $X$ it moves the person one space to the right.  If this causes the person to be standing in front of the nail they take one of three actions: 
	\begin{itemize} 
	\item if the nail is unoccupied, the person hangs picture $a$ on it; 
	\item if the nail holds picture $a$, the person removes it and leaves the nail unoccupied; 
	\item if the nail holds picture $b$, the person leaves the picture undisturbed. 
	\end{itemize} 
We claim that these rules can be implemented by a range $1$ block code and that $\varphi_a$ is invertible.  Similarly when $\varphi_b$ is applied to an element of $X$ it moves the person one space to the right.  If this causes the person to be standing in front of the nail, the analogous rules (with the roles of $a$ and $b$ reversed) apply.  This is also invertible and can be implemented by a range $1$ block code.  Note that $\varphi_a$ and $\varphi_b$ carry elements of $X$ to elements of $X$.  Finally suppose $w=(w_1,w_2,\dots,w_k)\in\{a,b\}^k$ and let $g\in\langle\varphi_a,\varphi_b\rangle^+$ be 
$$ 
g=g_1g_2\cdots g_k 
$$ 
where $g_i\in\{\varphi_a,\varphi_b\}$ for each $1\leq i\leq k$ is the automorphism corresponding to letter $w_i$.  First we show how to find $k$ by observing the action of $g$ on $X$.  For each $i\in\N$, let $x_i\in X$ be the configuration which has a $1$ at the origin, a $p$ exactly $i$ spaces to the left of the origin, and $0$'s elsewhere.  Note that $gx_k$ places the person at the origin and $gx_i$ places the person off the origin for all $i\neq k$.  Consequently, the length of a minimal presentation of $g$ by $\varphi_a$ and $\varphi_b$ can be deduced from this information and all representations of $g$ as a product of $\varphi_a$ and $\varphi_b$ (but not their inverses) have the same length.  Now fix $1\leq i\leq k$.  Then $gx_{k-i+1}$ is a configuration with the letter representing $g_i$ ($a$ or $b$) at the origin.  Therefore the natural surjection from $\{a,b\}^*$ to $\langle\varphi_a,\varphi_b\rangle^+$ is an injection and so this is the free semi-group of rank $2$.

\end{document}